\documentclass[11pt, a4]{article}
\usepackage{amssymb, amsfonts,amsmath, amsthm,amsbsy}
\usepackage{setspace,graphicx} 
\usepackage[all]{xy}  
\usepackage{hyperref}
\usepackage{enumerate} 

\newtheorem{theorem}{Theorem}[section]
\newtheorem{lem}[theorem]{Lemma}

\newtheorem{cor}[theorem]{Corollary}

\theoremstyle{definition}

\newtheorem{no}[theorem]{Notation}
\newtheorem{re}[theorem]{Remark}

\newcommand{\Div}[0]{\ensuremath{\textup{div}}}
\newcommand{\dom}[0]{\ensuremath{\textup{Dom}}}
\newcommand{\id}[0]{\ensuremath{\textup{Id}}}
\newcommand{\image}[0]{\ensuremath{\textup{Image}}}
\newcommand{\Ker}[0]{\ensuremath{\textup{Ker}}}

\newcommand{\sgn}[0]{\ensuremath{\textup{sgn}}}

\newcommand{\paral}{/\kern-0.3em/}
\def\parals_#1{/\kern-0.3em/_{\!#1}}
\def\paralss_#1^#2{/\kern-0.3em/_{\!#1}^{\!#2}}

\begin{document}
\begin{center}
\large{\bf{Generalised Clark-Ocone formulae for differential forms}}
\par
By \normalsize YANG Yuxin
\par
\today
\end{center}
\textbf{Abstract:}
We generalise the Clark-Ocone formula for functions to give analogous  representations for differential forms on the classical Wiener space.  Such formulae provide explicit expressions for closed and co-closed differential forms and, as a by-product, a new proof of the triviality of the $L^2$ de Rham cohomology groups on the Wiener space, alternative to Shigekawa's approach \cite{shigekawa1986rham} and the chaos-theoretic version \cite{yang2011ito}.  This new approach has the potential of carrying over to curved path spaces, as indicated by the vanishing result for harmonic one-forms in \cite{elworthy2011vanishing}.  For the flat path group, the generalised Clark-Ocone formulae can be derived using the It\^o map.
\textbf{Acknowledgements:} 
The author is grateful to Professor David Elworthy for suggesting the problem and for numerous valuable discussions.  
This research was partially funded by an Early Career Fellowship from the Warwick Institute of Advanced Study.
\section{Introduction}
\label{sec:intro}
The martingale representation theorem expresses any square integrable function on the classical Wiener space as the sum of its expectation and an It\^o integral.  If, in addition, the function is $H$-differentiable, with $H$ being the Cameron-Martin space,  the integrand of the It\^o integral can be expressed as the conditional expectation of the $H$-derivative.  The resultant representation is called the Clark-Ocone formula \cite{clark1970representation, ocone1984malliavin}.
\par
As one of the basic tools in stochastic analysis, the Clark-Ocone representation has many important applications and generalisations.  One of its crucial consequences is the spectral gap inequality obtained by S. Fang \cite{fang1994inegalite} for Riemannian path spaces; see \cite{capitaine1997martingale} for the derivation of other related functional inequalities, including the logarithmic Sobolev inequality first proved by Gross \cite{gross1975logarithmic, gross1991logarithmic} using different methods. 
We are interested in extending these results to the study of 
the Hodge-Kodaira Laplacian on differential forms, 
and we do so by generalising the Clark-Ocone formula for functions (zero-order forms)  to analogous representations for higher-order forms.  Such representations give explicit formulae for closed and co-closed forms; as a result, they offer a new proof of the vanishing of the $L^2$ de Rham cohomology groups on the classical Wiener space, alternative to the proof first given by I. Shigekawa \cite{shigekawa1986rham} and the chaos-theoretic version in \cite{yang2011ito}. 
\par
Shigekawa's result applies more generally to abstract Wiener spaces, which include the classical Wiener space as a special example.  It was the first step in developing a Hodge theory on infinite dimensional manifolds, a goal set by L. Gross \cite{gross1967potential} in his pioneering work on infinite dimensional potential theory. 
Shigekawa's definitive treatment of the linear case provides guidance for the study of nonlinear cases. For example, Fang and Franchi \cite{fang1997differentiable} used the It\^o map to transfer the vanishing result from the classical 
Wiener space to the path spaces over a compact Lie group with a bi-invariant metric. 
However, the problem of developing a Hodge theory in more general infinite-dimensional manifolds with curvature remains open.  
Our approach depends crucially on the filtration structure inherent in the classical Wiener space, hence does not extend to a general abstract Wiener space; on the other hand, it shows promise of carrying over to curved Riemannian path spaces as evidenced by our proof of the vanishing of the first $L^2\!$ cohomology there \cite{elworthy2011vanishing}. 
\par
The organisation of this article is as follows. 
In Section \ref{sec:notation} we review the basic notation and relate the Clark-Ocone formula to Shigekawa's result for zero-forms. 
In Section \ref{sec:CO-q} we present the generalised formulae for differential forms, with explicit expressions for closed and co-closed forms.  
Shigekawa's vanishing result for higher-order de Rham cohomology groups on the classical Wiener space can be seen as one of the direct consequences of these generalised Clark-Ocone formulae.  
The proofs of the main theorems are given in Section \ref{sec:proof}, and Section \ref{sec:extension} contains the extension of the generalised Clark-Ocone formulae to the path group via the It\^o map, which was shown to be a differentiable isomorphism by Fang and Franchi \cite{fang1997differentiable}.
\section{Notation and Motivation}
\label{sec:notation}
Fixing $T>0$, we denote by $C_0=C_0([0,T]; \mathbb{R}^m)$ the classical Wiener space, which is the collection of continuous functions from $[0,T]$ to $\mathbb{R}^m$ starting at the origin.   
This is a separable Banach space with the uniform norm 
\[
\|\sigma\|_{C_{0}} = \sup_{t \in [0,T]} \| \sigma (t) \|_{\mathbb{R}^{m}}. 
\]
Let $\gamma$ be the classical Wiener measure on $C_0$, and 
$H= L^{2,1}_0([0,T]; \mathbb{R}^m)$ the Cameron-Martin Hilbert space, which is the subspace of finite energy paths equipped with the inner product
\[
\langle h_{1}, h_{2} \rangle_{H} = \int_0^T \langle \frac{d}{dt}h_1, \frac{d}{dt}h_2  \rangle_{\mathbb{R}^{m}} dt, \quad h_1, h_2 \in H.
\]
The triple $(C_0, H, \gamma)$ gives the primary example of an abstract Wiener space, first studied by Gross \cite{gross1965abstract}.
\par
From the works of Gross \cite{gross1965abstract, gross1967potential}, it becomes clear that $H$-directional derivatives are the more natural objects to study than the usual Fr\'echet-derivatives.  Correspondingly we consider $H$-differential-forms, i.e., sections of dual bundle of exterior powers of $H$. Since we are primarily interested in the $L^2$ theory, we concentrate on $L^2$ and $\mathbb{D}^{2,k}$ forms, denoted by $L^2\Gamma(\wedge^q H)^*$ and $\mathbb{D}^{2,k}\Gamma(\wedge^q H)^*$, respectively, where $q\in \mathbb{N}$ is the order of the differential forms (while functions are considered zero-forms).
Throughout, we denote by $\otimes^q H$ the standard Hilbert completion of algebraic tensor products of $H$, by $\wedge^q H$ the $q$-fold skew-symmetric tensor products  completed using the Hilbert space cross norm inherited from $\otimes^q H$, and 
by $\mathbb{D}^{p,k}$ the Sobolev spaces defined using the $H$-derivatives and Hilbert-Schmidt norms, with $\mathbb{D}^{2,0}=L^2$. 
\par
The Clark-Ocone formula states that any function $F\in\mathbb{D}^{2,1}$ can be expressed as
\begin{equation}
\label{eq:CO_0}
F=\mathbb{E}F+\int_0^T\langle \mathbb{E}[\frac{d}{dt}(\nabla F)_t|\mathcal{F}_t],dB_t\rangle_{\mathbb{R}^m},
\end{equation}
where $\nabla F$ is the $H$-gradient obtained from the $H$-derivative $DF: C_0\rightarrow H^*$ via 
\begin{equation}
\label{eq:grad}
\langle \nabla F(\sigma), h\rangle_H = DF(\sigma)(h),\quad \sigma\in C_0, h\in H,
\end{equation}
and the filtration $\{\mathcal{F}_t\}_{t\in [0,T]}$ is generated by the Brownian motion $B$ on $\mathbb{R}^m$, with $B_t(\sigma)=\sigma(t)$ for $\sigma\in C_0$ and $t\in [0,T]$. We assume that all the sigma-algebras are completed with respect to $\gamma$.  
\par
Wu \cite{wu1990traitement} observed that the formula (\ref{eq:CO_0}) can be taken to hold for general $L^2$ functions, as the projection map from the space of $L^2$ processes onto its closed subspace of adapted processes has a special smoothing property.  Denote the subspace of all $L^2$ processes adapted to $\{\mathcal{F}_t\}_{t\in [0,T]}$ by $V$, and the projection onto it by $P_V$. 
Let $\delta$ be the adjoint of the gradient operator, which is the Skorohod integral and coincides with the It\^o integral on adapted processes.  
Wu \cite{wu1990traitement} reformulated the Clark-Ocone formula as 
\begin{equation}
\label{eq:CO_0_Wu}
F=\mathbb{E}F+\delta P_V\nabla F,
\end{equation}
while allowing for a more abstract and general interpretation of $V$.  
He also proved that $\delta (P_V\nabla) = P_{\delta(V)}$, and applied the It\^o isometry to show
 \[
 |(P_V\nabla)F|_{L^2(C_0;H)} = |\delta (P_V\nabla)F |_{L^2(C_0;\mathbb{R}^m)} 
 =  |P_{\delta (V)} F|_{L^2(C_0;\mathbb{R}^m)} \le |F|_{L^2(C_0;\mathbb{R}^m)}.
 \]
This means that the composed operator $(P_V\nabla)$ extends to a linear operator on $L^2$. 
We will make use of this observation in the sequel. 
\par
The representation (\ref{eq:CO_0}) shows that the operator $\nabla$ has a closed range and that the Laplacian $\Delta=\delta\nabla$ on the classical Wiener space has a spectral gap, based on a result by Donnelly \cite{donnelly1981differential}. 
We also observe that, since $\delta$ can be regarded as a (negative) divergence operator, the Clark-Ocone formula gives a novel solution to the equation $\Div(V)=F$, by expressing the unknown vector field $V$ in terms of the given function $F$.  In addition, it shows that 
\begin{description}
\item[$\mathbf{1}$.]  
$\nabla F=0 \iff F=\textup{constant} (=\mathbb{E}F)$; and 
\item[$\mathbf{2}$.] 
$\mathbb{E}F=0 \iff F\in \image(\delta)$.
\end{description}
These are precisely the result of Shigekawa \cite{shigekawa1986rham} for $L^2$ functions considered as zero-forms, 
i.e.,  the only harmonic zero-forms are constant functionals. 
In fact, the formula (\ref{eq:CO_0}) gives an explicit Hodge decomposition for zero-forms in the form of 
\[
F = \textup{constant} + \delta(v),
\]
and provides the expressions for $v$ and the constant in terms of $F$.  
\par
This motivated our attempt to find similar expressions which would imply the corresponding result for higher-order forms on the classical Wiener space.  
Recall the definitions of exterior derivative $d_q$ and its adjoint $d_q^*$ in Shigekawa \cite{shigekawa1986rham}:
\begin{equation}
\label{eq:dq}
d_q=(q+1)A_{q+1}D, 
\quad \mbox{ and }\quad
d_{q}^*=D^*.
\end{equation} 
Here $D^*$ is the adjoint of $D$, and $A_q: L(\otimes^q H; \mathbb{R})\rightarrow  L(\wedge^q H; \mathbb{R})$ is  the alternating map defined by 
\[
A_q\phi(h_1, \cdots, h_q)=\frac{1}{q!}\sum_{\rho\in \mathfrak{S}_q}\sgn(\rho) \phi(h_{\rho(1)}, \cdots,h_{\rho(q)}), 
\]
where $\phi\in L(\otimes^q H; \mathbb{R}), \, h_1,\cdots,h_q\in H$, and the summation is over all $q!$ elements of the symmetric group $\mathfrak{S}_q$, which consists of all permutations of $\{1, \cdots,q\}$.  
Shigekawa's convention for the wedge product is given by 
\[
h_1\wedge\cdots\wedge h_q = q! \,A_{q}(h_1\otimes\cdots\otimes h_q), \quad h_1, \cdots, h_q\in H, 
\]
and $d_q^*$ is dual to $d_q$ with respect to the following inner product on $\wedge^q H$
\[
\langle h_1\wedge\cdots\wedge h_q, g_1\wedge\cdots\wedge g_q\rangle_{\wedge^q H} 
= \det(\langle h_i, g_j\rangle_H), \quad h_1, \cdots, h_q, g_1, \cdots, g_q \in H. 
\]
Note that $d_{q+1}d_q=0$ and $d_q^*d_{q+1}^*=0$.  
The $L^2$ domain $\dom(d_q)$ of $d_q$ is obtained by taking the $L^2$ completion of cylindrical forms.  
In what follows, we denote the closure of the closable operators $D$ and $d_q$ by the same symbols. 
\par
Let $\Delta_q=d_q^*d_q+ d_{q-1}d_{q-1}^*$ be the Hodge-Kodaira Laplacian on $q$-forms, and $\mathfrak{h}_q$ the set of all the harmonic forms of degree $q$, i.e., $\phi\in L^2\Gamma (\wedge^q H)^*$ is in $\mathfrak{h}_q$ if $\phi\in\dom(\Delta_q)$ and $\Delta_q\phi = 0 $.  
\begin{theorem}[Shigekawa \cite{shigekawa1986rham}]
\label{th:shigekawa}
$L^2\Gamma (\wedge^q H)^*= \image(d_{q-1})\oplus \image(d_q^*)\oplus\mathfrak{h}_q$, 
where 
\begin{description}
\item[1.] $\image(d_{q-1}) = \Ker(d_q)$;  
\item[2.] $\image(d_q^*) = \Ker(d_{q-1}^*)$;  
\item[3.] 
$\mathfrak{h}_q=\{ 0 \}$ for $q\ge 1$, and $\mathfrak{h}_0=\{\mbox{constant functions}\}$.
\end{description}
\end{theorem}
To emulate the case of zero-forms, we seek  the following representations for  $L^2$ $H$-$q$-forms, $q\in \mathbb{N}$: 
\begin{equation}
\label{eq:Hodge_q}
\phi = d_{q-1}\psi + M_q(d_q\phi), \quad\forall \phi\in 
\dom(d_q),
\end{equation}
and
\begin{equation}
\label{eq:Hodge_q*}
\phi = d_q^*\theta + N_q(d_{q-1}^*\phi), \quad\forall \phi\in 
\dom(d_{q-1}^*),
\end{equation}
where $M_q$ and $N_q$ are nice linear functions. 
Such expressions imply the theorem above since, for any $L^2$ $H$-$q$-form $\phi$, 
\begin{description}
\item[1.]
$d_q\phi = 0 \iff \phi=d_{q-1}\psi,\mbox{ some }\psi\in \dom(d_{q-1})\subset L^2\Gamma(\wedge^{(q-1)} H)^*$; and
\item[2.]
$d_{q-1}^*\phi = 0 \iff \phi=d_q^*\theta,\mbox{ some } \theta\in \dom(d_q^*) \subset L^2\Gamma(\wedge^{(q+1)} H)^*$.
\end{description} 
\par
To proceed, we note first that the Riesz representation theorem gives a natural isomorphism between the $L^2$ $H$-forms and $L^2$ skew-symmetric $H$-vector-fields $L^2\Gamma(\wedge^q H)$. 
Some of our expressions are more conveniently written in terms of vector fields, and we switch between differential forms and skew-symmetric vector fields using the above isomorphism and the following 
\begin{no}
Given any $u \in L^2\Gamma (\wedge^q H)^*$, we define $u^\sharp \in L^2\Gamma (\wedge^q H)$ by 
\[
u(h)=\langle u^\sharp, h\rangle_{ \wedge^q H} ,\quad h\in L^2\Gamma (\wedge^q H).
\]
Similarly for $u \in L^2\Gamma (\wedge^q H)$, we define  $u^\flat\in L^2\Gamma (\wedge^q H)^*$ by 
\[
u^\flat(h)=\langle u, h\rangle_{ \wedge^q H},\quad h\in L^2\Gamma (\wedge^q H).
\]
\end{no}
\begin{no}
Corresponding to the exterior derivative $d_q$, we define an operator
$d_{q}^\sharp:\dom(d_{q}^\sharp)\subset L^2\Gamma (\wedge^q H)\rightarrow L^2\Gamma (\wedge^{(q+1)} H)$ 
on skew-symmetric $q$-vector fields by  
\[
d_{q}^\sharp u=(d_q u^\flat)^\sharp, \quad u\in L^2\Gamma (\wedge^q H)
\]
where $u \in \dom(d_{q}^\sharp)$ iff $u^\flat\in \dom(d_q)$. 
Similarly for $d_q^*$, we define  
\[
d_{q}^{*\sharp} u=(d_{q}^* u^\flat)^\sharp, \quad u\in L^2\Gamma (\wedge^{(q+1)} H),
\]
where $u \in \dom(d_q^{*\sharp})$ iff $u^\flat\in \dom(d_q^{*})$. 
We note that $d_q^{*\sharp} = d_q^{\sharp*}$, and from (\ref{eq:dq}) we obtain 
\[
d_q^\sharp =  (q+1)A_{q+1}\nabla, \quad \mbox{ and } \quad 
d_q^{*\sharp} = \delta|_{L^2\Gamma (\wedge^{(q+1)} H)}.
\]
Note that for a vector-valued $\mathbb{D}^{2,1}$ function $F:C_0\rightarrow X$, where $X$ is a separable Hilbert space, the $H$-gradient $\nabla F: C_0\rightarrow H\otimes X$ is defined the same way as in (\ref{eq:grad}).  
Similarly, the adjoint map 
\[
\delta: \dom(\delta)\subset L^2\Gamma(H\otimes X) \rightarrow L^2\Gamma(X)
\] 
still coincides with the Skorohod integral. 
\end{no}
\begin{no}
As an analogue of the interior product, we define  
\[
{}_{(i)}\!\langle h_1 \otimes \cdots \otimes h_k, h_{k+1}\rangle_H = \langle h_i, h_{k+1}\rangle_H h_1\otimes\cdots\otimes \hat{h}_i\otimes\cdots\otimes h_k, 
\]
where $h_j\in H$ for $j=1$ to $k+1$, and $\hat{h}_i$ indicates the omission of $h_i$.   When $i=1$, we recover the standard interior product: 
\begin{eqnarray*}
{}_{(1)}\!\langle h_1\wedge \cdots \wedge h_k, h_{k+1}\rangle_H 
&=& \sum_{j=1}^k (-1)^{j-1} \langle h_j, h_{k+1}\rangle_H h_1\wedge \cdots\wedge\hat{h}_j\wedge \cdots \wedge h_k \\
&=&\iota_{h_{k+1}} (h_1 \wedge \cdots \wedge h_k).
\end{eqnarray*}
\end{no}
\par
Let $ev_s: C_0\rightarrow \mathbb{R}^m$ be the evaluation map at time $s\in[0,T]$, i.e, for any $u\in C_0$, we have $ev_s(u)=u_s$. For a tensor $u\in \otimes^q_\epsilon C_0$, with $\otimes^q_\epsilon C_0$ denoting the injective tensor product space of $C_0$, we can make use of the isometry 
\[
\mathbf{i}: \otimes^q_\epsilon C_0\rightarrow C_0([0,T]^q; \otimes^q \mathbb{R}^m),
\]
where 
 $C_0([0,T]^q; \otimes^q \mathbb{R}^m)$ consists of continuous functions $\sigma : [0, T ]^q \rightarrow\otimes^q \mathbb{R}^m$ such that $\sigma(t_1, \cdots, t_q)= 0$ if $t_j = 0$ for any integer $j$ between $1$ and $q$ (see \cite{elworthy2008l2} for a more detailed description). So we have
\[
u_{s_1, \cdots, s_q} =\mathbf{i}(u)(s_1, \cdots, s_q) = (ev_{s_1}\otimes\cdots\otimes ev_{s_q} )u. 
\]
We also recall the isomorphism between the Hilbert spaces $L^{2}([0,T]; \mathbb{R}^{m})$ and $H=L^{2,1}_0([0,T]; \mathbb{R}^{m})$ given by the indefinite integral 
\[
\int_0^.: L^{2}([0,T];\mathbb{R}^{m})\rightarrow H, 
\]
with the inverse map 
\[
\frac{d}{d.}: H\rightarrow L^{2}([0,T];\mathbb{R}^{m}).
\]
This gives rise to an isometry between the Hilbert spaces of the tensor powers $\otimes^qH$ and $\otimes^qL^2([0, T ]; \mathbb{R}^m)\cong L^2([0, T ]^q;\otimes^q \mathbb{R}^m)$.  Therefore, $u\in \otimes^q H$ iff  
\[
u_{s_1, \cdots, s_q} 
= \int_0^{s_1}\cdots\int_0^{s_q} \frac{\partial^{q}}{\partial r_1 \cdots \partial r_q}u_{r_1, \cdots,r_q}dr_1\cdots dr_q. 
\]
\begin{no}
\label{no:partial}
We use the shorthand notation 
\[
\partial^q_{r_1, \cdots, r_q} u= \frac{\partial^{q}}{\partial r_1 \cdots \partial r_q}u_{r_1, \cdots,r_q}
\]
from now on. 
We also write  $D_r u =\frac{d}{dr}( \nabla u )_r$ for $u \in \mathbb{D}^{2,1}$, following Nualart's notation \cite{nualart2006malliavin}. 
\end{no}
For example, we can write, for $u\in\dom(d_q^\sharp)$, a.e. $s_1, \cdots, s_{q+1}\in[0,T]$,   
\begin{eqnarray}
\label{eq:CWS_dq_B}
& &\partial^{q+1}_{s_1, \cdots, s_{q+1}} (d_q^{\sharp} u)\notag\\
&=& \sum_{j=1}^{q+1}(-1)^{j-1}D_{s_j}\partial^q_{s_1, \cdots,\hat{s}_j,\cdots, s_{q+1}} u\\ 
\label{eq:CWS_dq_C}
&=& 
D_{s_{1}} \partial^q_{s_2, \cdots,s_{q+1}} u
- \sum_{j=1}^{q}\tau_{1, j+1} D_{s_j}\partial^q_{s_2, \cdots,s_{j-1}, s_{1}, s_{j+1},\cdots, s_{q+1}} u.
\end{eqnarray}
Here $\tau_{i,j}: \otimes^n H\rightarrow \otimes^n H$ is the transposition operator, which acts by exchanging the $i$-th and $j$-th components of a tensor; that is, given any $h_1, \cdots, h_n \in H$, 
\[
\tau_{i,j}(h_1\otimes\cdots\otimes h_i\otimes\cdots\otimes h_j\otimes\cdots\otimes h_n)
=h_1\otimes\cdots\otimes h_j\otimes\cdots\otimes h_i\otimes\cdots\otimes h_n.
\]  
We write simply $\tau$ when it acts on a two-tensor, and often omit it when the indices of a tensor product give a clear indication of the ordering. 
In particular, for $q=1$, we have
\begin{equation}
\label{eq:d1_sharp_u}
\partial^2_{s_1, s_2}(d_1^{\sharp} u) = D_{s_1}(\frac{d}{d_{s_2}}u) - \tau [D_{s_2}(\frac{d}{d_{s_1}}u)].
\end{equation}
\section{Generalised Clark-Ocone Formulae}
\label{sec:CO-q}
Let $q\in \mathbb{N}$, and set $\bar{s} =\max_{i=1}^q s_i$ for $s_1, \cdots, s_q\in\mathbb{R}$.
\begin{theorem}[Clark-Ocone Formula for q-Forms: I]
\label{th:CO_q}
If $u\in \dom(d_q)$, the skew-symmetric $(q-1)$-vector-field 
$T_{q-1}(u)\in L^2\Gamma(\wedge^{(q-1)} H)$ defined by 
\begin{equation}
\label{eq:T_q-1}
T_{q-1}(u)
=\int_0^.\!\!\cdots\!\!\int_0^.\int_{\max_{i=2}^{q} r_i}^T \!\!{}_{(1)}\!\langle \mathbb{E}(\partial^q_{r_1, \cdots, r_q} u^\sharp |\mathcal{F}_{r_1}),dB_{r_1}\rangle_{\mathbb{R}^{m}} dr_2\cdots dr_{q}, 
\end{equation}
is in the domain of $d_{q-1}^{\sharp}$, and, for a.e. $s_1, \cdots, s_q\in[0,T]$,  
\begin{equation}
\label{eq:CO_q}
\partial^q_{s_1, \cdots, s_q} u^\sharp =
\partial^q_{s_1, \cdots, s_q}[d_{q-1}^{\sharp}T_{q-1}(u)]
+\int_{\bar{s}}^T\!\!\!{}_{(1)}\!\langle \mathbb{E}[\partial^{q+1}_{r, s_1, \cdots, s_q} (d_q u)^\sharp   |\mathcal{F}_r], dB_r\rangle_{\mathbb{R}^{m}}.
\end{equation}
Moreover, if $u\in \mathbb{D}^{2,k}\Gamma(\wedge^q H)^*$, we have $T_{q-1}(u)\in \mathbb{D}^{2,k+1}\Gamma(\wedge^{(q-1)} H)$. 
\end{theorem}
\begin{re}
For the case of $q=1$, we have 
\[
\frac{d}{ds}u_{s}^\sharp
=D_s\int_{0}^T\langle \mathbb{E}(\frac{d}{d r}u_{r}^\sharp  |\mathcal{F}_r), dB_r\rangle_{\mathbb{R}^{m}} +
\int_{s}^T\!\!\!{}_{(1)}\!\langle \mathbb{E}[\frac{\partial^2}{\partial r\partial s}(d_1u)^\sharp   |\mathcal{F}_r], dB_r\rangle_{\mathbb{R}^{m}}.
\]
This generalises to Riemannian path spaces and proves that there exist no harmonic $L^2$ one-forms there \cite{elworthy2011vanishing}.
\end{re}
\begin{cor}[Closed q-Forms]
\label{cor:dq_kernel}
$u\in\Ker(d_q)\implies u= d_{q-1} [T_{q-1}(u)^\flat]$, where the skew-symmetric $(q-1)$-vector-field $T_{q-1}(u)\in\dom(d_q^\sharp)$ is defined in (\ref{eq:T_q-1}). That is, any $u\in\Ker(d_q)$ can be expressed as follows: for a.e. $s_1, \cdots, s_q\in[0,T]$,  
\begin{eqnarray*}
& & \partial^q_{s_1, \cdots,s_q}u^\sharp\notag\\
&=& \partial^q_{s_1, \cdots,s_q} [d_{q-1}^{\sharp}T_{q-1}(u)]\\
&=&\mathbb{E}(\partial^q_{s_1, \cdots, s_q} u^\sharp | \mathcal{F}_{\bar{s}}) + \sum_{j=1}^q \int_{\bar{s}}^T \!\!\!{}_{(j+1)}\!\langle 
\mathbb{E}(D_{s_j} \partial^q_{s_1, \cdots,s_{j-1}, r,s_{j+1}, \cdots, s_q} u^\sharp|\mathcal{F}_{r}),dB_{r}\rangle_{\mathbb{R}^{m}}.\notag
\end{eqnarray*}
\end{cor}
\begin{cor}
$\image(d_{q})= \Ker(d_{q+1})$, so the image of $d_q$ is closed. 
\end{cor}
\begin{theorem}[Clark-Ocone Formula for q-Forms: II]
\label{th:CO_q*}
If $u\!\in\! \dom(d_{q-1}^*)$,  
the skew-symmetric $(q+1)$-vector-field 
$S_{q+1}(u)\in L^2\Gamma (\wedge^{(q+1)} H)$ defined by 
\begin{eqnarray}
\label{eq:S_q+1}
& & S_{q+1}(u)\\
&=& 
\int_0^.\!\!\cdots\!\!\int_0^.
\{\mathbb{E}[\mathbf{1}_{(\bar{r}, T]}(r)D_r\partial^q_{r_1, \cdots, r_q}u^\sharp |\mathcal{F}_{r}]  \notag\\
& &- 
\sum_{j=1}^q\mathbb{E}[
\mathbf{1}_{r_j=\max(r,\bar{r})}
\tau_{1,j+1}(D_{r_j}\partial^q_{r_1, \cdots,r_{j-1}, r,r_{j+1},\cdots, r_q} u^\sharp) |\mathcal{F}_{r_j}] \}
\,dr dr_1\cdots dr_{q}
\notag
\end{eqnarray}
is in the domain of $d_q^{*\sharp}$, 
and,  for a.e. $s_1, \cdots, s_q\in[0,T]$,  
\begin{eqnarray}
\label{eq:CO_q*}
\partial^q_{s_1, \cdots, s_q} u^\sharp 
&=&\partial^q_{s_1, \cdots,s_q}[d_q^{*\sharp}S_{q+1}(u)]  \notag\\
& & +
\sum_{j=1}^q(-1)^{j-1}\mathbf{1}_{s_j=\bar{s}}\,\mathbb{E}[D_{s_j} \partial^{q-1}_{s_1, \cdots,\hat{s}_{j}, \cdots, s_q} (d_{q-1}^* u)^\sharp|\mathcal{F}_{s_j}]
\end{eqnarray}
\end{theorem}
\begin{re}
For the case of $q=1$, we have, for a.e. $ s\in[0,T]$, 
\begin{eqnarray*}
\frac{d}{ds} u_s^\sharp 
&\!=\!&
\frac{d}{ds} d_1^{*\sharp}[ \int_0^.\!\!\int_{r_1}^.\mathbb{E}(D_{r}\frac{d}{dr_1} u_{r_1}^\sharp  |\mathcal{F}_{r}) drdr_1
-  \int_0^.\!\!\int_{r}^.\mathbb{E}(D_{r_1}\frac{d}{dr}  u_{r}^\sharp|\mathcal{F}_{r_1}) dr_1dr]\\
& &+\,\mathbb{E}[D_s \int_0^T\langle\frac{d}{dr} u_r^\sharp, dB_r\rangle |\mathcal{F}_s] .
\end{eqnarray*}
\end{re}
\begin{cor}[Co-closed q-Forms]
\label{cor:dq-1*_kernel} 
$u\in\Ker(d_{q-1}^*)\!\!\implies\!\!u  = d_q^*[S_{q+1}(u)^\flat]$, where the skew-symmetric vector field $S_{q+1}(u)\in\dom(d_q^{*\sharp})$ is given by 
(\ref{eq:S_q+1}).  That is, any $u\in\Ker(d_{q-1}^*)$ can be expressed as 
\begin{eqnarray*}
\partial^q_{s_1, \cdots,s_q}u^\sharp
&=&\partial^q_{s_1, \cdots,s_q}[d_q^{*\sharp}S_{q+1}(u)] \\
&=&
\int_{\bar{s}}^T\!\!\!{}_{(1)}\!\langle \mathbb{E}(D_r\partial^q_{s_1, \cdots, s_q} u^\sharp | \mathcal{F}_{r}),dB_{r}\rangle_{\mathbb{R}^{m}}\notag \\
& &  
- \sum_{j=1}^q \mathbf{1}_{s_j=\bar{s}}\int_0^{s_j}\!\!\!\!{}_{(j+1)}\!\langle 
 \mathbb{E}(D_{s_j}\partial^q_{s_1, \cdots,s_{j-1}, r,s_{j+1}, \cdots, s_q}u^\sharp|\mathcal{F}_{s_j}),dB_{r}\rangle_{\mathbb{R}^{m}}.\notag
\end{eqnarray*}
\end{cor}
\begin{re}
For the case of $q=1$, we have the following representation for `divergence-free' vector fields: i.e., $\Div(u) =0$ implies that,  for a.e. $ s\in[0,T]$, 
\[
\frac{d}{ds} u_s^\sharp 
= \int_s^T\!\!\!{}_{(1)}\!\!<\mathbb{E}(D_r\frac{d}{ds} u_s^\sharp  |\mathcal{F}_r),dB_r>_{\mathbb{R}^m}
- \int_0^s\!\!\!\!{}_{(2)}\!\!<\mathbb{E}(D_s\frac{d}{dr} u_r^\sharp  |\mathcal{F}_s),dB_r>_{\mathbb{R}^m}.
\]
\end{re}
\begin{re}
Corollary \ref{cor:dq-1*_kernel} follows directly from Theorem \ref{th:CO_q*}.  It can also be taken as a consequence of Theorem \ref{th:CO_q} by a duality argument:  given any $v\in \dom(d_{q})$,
we have 
\[
\mathbb{E}\langle u^{\sharp},d_{q-1}^{\sharp}T_{q-1}(v)\rangle
= \mathbb{E} \langle (d_{q-1}^{*}u)^{\sharp},T_{q-1}(v)\rangle= 0,
\] 
so applying (\ref{eq:CO_q}) to $v$, we obtain 
$u^{\sharp} = d_q^{*\sharp}S_{q+1}(u)$ since 
\begin{eqnarray*}
& &\mathbb{E}\langle u^{\sharp}, v^{\sharp} \rangle_{\wedge^q H}  \\
&=\!\!&\mathbb{E} \langle u^{\sharp}, \int_0^.\!\!\cdots\!\! \int_0^.\int_{\bar{s}}^T\!\!\!{}_{(1)}\!\langle \mathbb{E}[\partial^{q+1}_{r, s_1, \cdots, s_q}(d_q v)^{\sharp}   |\mathcal{F}_r], dB_r\rangle_{\mathbb{R}^{m}} ds_1\cdots ds_q \rangle_{\wedge^q H}  \\ 
&=\!\!&\mathbb{E} \int_0^T\!\!\cdots\!\! \int_0^T\!\!\langle  \mathbf{1}_{(\bar{s}, T]}(r)\mathbb{E}[D_r( \partial^q_{s_1, \cdots, s_q} u^{\sharp})|\mathcal{F}_r],\partial^{q+1}_{r, s_1, \cdots, s_q} (d_q v)^{\sharp}  \rangle
dr  ds_1\cdots ds_q\\
&=\!\!&\mathbb{E} \langle  S_{q+1}(u) , (d_q v)^{\sharp}  \rangle_{\wedge^{(q+1)} H} .
\end{eqnarray*}
\end{re}
\begin{cor}
$\image(d_q^*)= \Ker(d_{q-1}^*)$, so the image of $d_q^*$ is closed. 
\end{cor}
\begin{re}
\label{re:P_V}
As mentioned in Section \ref{sec:notation}, the integrand in the Clark-Ocone formula (\ref{eq:CO_0}) can be regarded as the projection of $\nabla F$ onto the space of adapted processes. A similar interpretation applies to our generalised Clark-Ocone formulae for higher order forms. We define a subspace of $L^2\Gamma(\otimes^q H)$
\[
V^{(q)}=\{u\in L^2\Gamma(\otimes^q H): u_{s_1, \cdots, s_q}\in \mathcal{F}_{\bar{s}},  \mbox{ a.e. } s_1, \cdots, s_q\in [0,T]\}.
\]
Let $P_{V^{(q)}}$ be the projection onto $V^{(q)}$ defined by 
\begin{eqnarray*}
P_{V^{(q)}} u
&=& \int_0^.\cdots \int_0^. \mathbb{E}(\partial^q_{s_1, \cdots, s_q}u  |\mathcal{F}_{\bar{s}})ds_1\cdots ds_q\\
&=& \sum_{j=1}^{q} \int_0^.\cdots \int_0^.\mathbf{1}_{s_j=\bar{s}}\,\mathbb{E}(\partial^q_{s_1, \cdots, s_q}u |\mathcal{F}_{s_j})ds_1\cdots ds_q, 
\end{eqnarray*}
and $P^j_{V^{(q)}}$ the $j$-th term in the above sum, i.e., for $j=1$ to $q$, 
\[
P^j_{V^{(q)}} u = \int_0^.\cdots \int_0^.\mathbf{1}_{s_j=\bar{s}}\,\mathbb{E}(\partial^q_{s_1, \cdots, s_q}u |\mathcal{F}_{s_j})ds_1\cdots ds_q. 
\]
Now we can state the generalised Clark-Ocone formulae (\ref{eq:CO_q}) and (\ref{eq:CO_q*}) as
\begin{equation}
\label{eq:CO_q_P_V}
u^\sharp  = d_{q-1}^\sharp\delta P^1_{V^{(q)}}u^\sharp+\delta P^{1}_{V^{(q+1)}}(d_qu)^\sharp,
\end{equation}
and 
\begin{equation}
\label{eq:CO_q*_P_V}
u^\sharp =d_q^{*\sharp}A_{q+1}P^{1}_{V^{(q+1)}}\nabla u^\sharp + A_qP^1_{V^{(q)}}\nabla (d_{q-1}^* u)^\sharp.
\end{equation}
Wu's formulation (\ref{eq:CO_0_Wu}) can be regarded as the special case of (\ref{eq:CO_q_P_V}) for $q=0$.  
We note that $\delta$ is injective on the image of $P_{V^{(q)}}^1$ for any $q\in \mathbb{N}$, where it coincides with the standard It\^o integral. 
\end{re}
\begin{re}
\label{re:iff}
It is interesting to observe from (\ref{eq:CO_q})  that the apparently weaker condition 
\[
\mathbf{1}_{(\bar{s}, T]}(r)\mathbb{E}[\partial^{q+1}_{r, s_1, \cdots, s_q}(d_q u)^\sharp   |\mathcal{F}_r] = 0, \quad\mbox{ a.e. }r, s_1, \cdots, s_q\in[0,T]
\]
is actually equivalent to the apparently stronger condition $d_q u=0$.   In fact, by skew-symmetry we also have
\[
\mathbb{E}[\partial^{q+1}_{r, s_1, \cdots, s_q}(d_q u)^\sharp   |\mathcal{F}_{\max(r, \bar{s})}] = 0\iff
d_q u=0, \quad\mbox{ a.e. }r, s_1, \cdots, s_q\in[0,T].
\]
This is in line with the situation for $q=0$, where the formula (\ref{eq:CO_0}) implies 
\[
\nabla F=0 \iff F=\textup{constant} \iff \mathbb{E}(D_rF  |\mathcal{F}_r)=0, \mbox{ a.e. }r \in[0,T].
\]
\end{re}
\begin{re} 
\label{re:iff_P_V}
We can restate Remark \ref{re:iff} in the notation of Remark \ref{re:P_V}: for $q\in\mathbb{N}\cup\{0\}$, $u\in\dom(d_q)$, and any integer $j$ between $1$ and $q$,  
\[
P^j_{V^{(q+1)}}(d_q u)^\sharp=0 \iff 
P_{V^{(q+1)}}(d_q u)^\sharp =0 \iff 
d_q u=0.
\]
\end{re}
\begin{re} 
\label{re:unique}
Remark \ref{re:iff_P_V} implies
\[
\Ker (P_{V^{(q+1)}}) \cap \image(d_q^\sharp)=\{0\}.
\]
It is easy to check that the second term on the right-hand side of (\ref{eq:CO_q_P_V}) lies in $\Ker (P_{V^{(q)}})$. 
Therefore, the expression (\ref{eq:CO_q_P_V}) actually gives a unique decomposition of a $q$-$H$-form $u$ in the form of
\begin{equation}
\label{eq:unique_q}
u= d_{q-1}v + w, \quad w^\sharp\in \Ker(P_{V^{(q)}}).
\end{equation}
Since $w^\sharp$ is expressed as an It\^o integral, the integrand is also uniquely given.
The fact that the Clark-Ocone formula  (\ref{eq:CO_0}) gives a unique representation of the function $F$ as the sum of a constant and an It\^o integral can be seen as the special case for $q=0$.  
\end{re}
\begin{re} 
By duality, we also have, cf. Remark \ref{re:unique}, 
\[
V^{(q)} \cap \image(d_q^{*\sharp})=\{0\}.
\]
Similarly to (\ref{eq:unique_q}), we obtain from (\ref{eq:CO_q*_P_V}) another unique decomposition of a $q$-$H$-form $u$: 
\begin{equation}
\label{eq:unique_q*}
u= d_{q}^*v + w, \quad w^\sharp\in V^{(q)}.
\end{equation}
If we take $P_{V^{(0)}}$ as the projection onto constants (by taking expectation),  decomposition (\ref{eq:unique_q*}) also incorporates the Clark-Ocone formula  (\ref{eq:CO_0}) as the special case for $q=0$. 
We cannot, however, say much about the uniqueness of the integrand in general, as Skorohod integrals are involved for $q>0$.
\end{re}
\begin{re}
It is worth pointing out that the two terms in the equation (\ref{eq:CO_q}) are not orthogonal to each other in general; similarly for (\ref{eq:CO_q*}).  These generalised Clark-Ocone formulae do not give an explicit Hodge decomposition in the form of 
\begin{equation}
\label{eq:Hodge}
\phi= d_{q-1}\psi + d_q^*\theta + h,
\end{equation}
where the harmonic component $h=0$.  
\end{re}
The main ingredient in the proofs of Theorems \ref{th:CO_q} and \ref{th:CO_q*} is the well-known commutation relationship between the derivative and divergence operators, which can be most concisely expressed in the form of a Heisenberg commutation relationship as 
$[\nabla, \delta]=\id_H$. 
Or, as Nualart \cite{nualart2006malliavin} (Proposition 1.3.8) puts it,  for any $u\in \mathbb{D}^{2,1}$ such that $\tau_{1,2} \nabla u \in \dom(\delta)$, 
\begin{equation}
\label{eq:commutation}
D_t\delta u = \frac{d}{dt}u_t + \int_0^T\!\!\!{}_{(2)}\!\langle D_t \frac{d}{ds} u_s, dB_s\rangle_{\mathbb{R}^{m}}. 
\end{equation}
This generalises to any vector field $u\in\mathbb{D}^{2,1}\Gamma(\otimes^n H)$ satisfying the same condition that $\tau_{1, 2} \nabla u\in \dom(\delta)$, and we have
\begin{eqnarray}
\label{eq:commutation_tensor}
& &D_{t}\!\!\int_0^T \!\!\!\!{}_{(1)}\!\langle  \partial^q_{r, {s_1}, \cdots, s_{n-1}}u ,   dB_{r}\rangle_{\mathbb{R}^m} \notag\\
&=& \partial^q_{t, s_1, \cdots,s_{n-1}} u +
\int_0^{T}\!\!\!\!{}_{(2)}\!\langle D_{t} \partial^q_{r, s_1, \cdots,s_{n-1}} u, dB_r\rangle_{\mathbb{R}^{m}}.
\end{eqnarray}
\par
We state some useful consequences of the commutation formula.  First, since $d_q^{*\sharp}=\delta$ on skew-symmetric tensor fields, we apply (\ref{eq:commutation_tensor}) to arrive at 
\begin{lem}[Commutation Formula for $d_q^*$]
\label{lem:commutation_dq*}
Suppose $u\in\mathbb{D}^{2,1}\Gamma(\wedge^q H)$ and 
$\nabla u\in\dom(\delta)$.  Then we have $d_{q-1}^* u \in\mathbb{D}^{2,1}$, and 
\[
\nabla d_{q-1}^{*\sharp} u= u+ \delta (\tau_{1,2}\nabla u).
\]
\end{lem}
\begin{lem}[Commutation Formula for $d_q$]
\label{lem:commutation_dq}
Suppose $u\in\mathbb{D}^{2,1}\Gamma(\wedge^q H)$ satisfies 
$\tau_{1,j+1}\nabla u \in \dom(\delta)$, for all $j=1$ to $q$. Then $\delta u \in \dom(d_{q-1}^{\sharp})$, and 
\[
 d_{q-1}^{\sharp}\delta u= q \,  u+ \sum_{j=1}^{q} \delta (\tau_{1, j+1}\nabla u). 
\]
\end{lem}
\begin{proof}
This follows from (\ref{eq:CWS_dq_B}) and (\ref{eq:commutation}) since, for a.e. $s_1, \cdots, s_q \in[0,T]$, 
\begin{eqnarray*}
& &  
 \partial^q_{s_1, \cdots, s_q} [d_{q-1}^{\sharp}\int_0^.\cdots \int_0^.\int_0^T \!\!\!{}_{(1)}\!\langle \partial r^q_{r_1, \cdots, r_q} u ,dB_{r}\rangle_{\mathbb{R}^{m}} dr]\notag\\
&=& 
\sum_{j=1}^q (-1)^{j-1}D_{s_j}[\int_0^T \!\!\!{}_{(1)}\!\langle \partial^q_{r, s_1, \cdots,\hat{s}_{j}, \cdots, s_q} u ,dB_{r}\rangle_{\mathbb{R}^{m}} ]\notag\\
&=& 
\sum_{j=1}^qD_{s_j}[\int_0^T \!\!\!{}_{(j)}\!\langle \partial^q_{s_1, \cdots,s_{j-1}, r,s_{j+1}, \cdots, s_q} u ,dB_{r}\rangle_{\mathbb{R}^{m}} ]\notag\\
&=& 
q\, \partial^q_{s_1, \cdots, s_q} u +
\sum_{j=1}^{q} \int_0^T \!\!\!{}_{(j+1)}\!\langle D_{s_j} \partial^q_{s_1, \cdots,s_{j-1}, r,s_{j+1}, \cdots, s_q} u,dB_{r}\rangle_{\mathbb{R}^{m}}. \qedhere
\end{eqnarray*}
\end{proof}
\begin{re}
For $q=1$, both Lemmas \ref{lem:commutation_dq} and \ref{lem:commutation_dq*} reduce to (\ref{eq:commutation}).
\end{re}
\section{Proofs}
\label{sec:proof}
\begin{proof}[\textbf{Proof of Theorem \ref{th:CO_q}}]
We first prove the result for $u\in \mathbb{D}^{2,1}\Gamma(\wedge^q H)^*$, and then use an approximation argument to extend to a general $u\in \dom(d_q)$.  
\par
We apply the Clark-Ocone formula to write, for a.e. $s_1, \cdots, s_q\in[0,T]$,  
\[
\partial^q_{s_1, \cdots, s_q} u^\sharp 
= \mathbb{E}(\partial^q_{s_1, \cdots, s_q} u^\sharp )
+ \int_0^T\!\!\!{}_{(1)}\!\langle \mathbb{E}[D_r\partial^q_{s_1, \cdots, s_q} u^\sharp | \mathcal{F}_r], dB_r\rangle_{\mathbb{R}^{m}}.
\]
Taking conditional expectation with respect to $\mathcal{F}_{\bar{s}}$, we see
\begin{equation}
\label{eq:conditioned_CO_qform}
\mathbb{E}(\partial^q_{s_1, \cdots, s_q} u^\sharp | \mathcal{F}_{\bar{s}})
= \mathbb{E}(\partial^q_{s_1, \cdots, s_q} u^\sharp )
+ \int_0^{\bar{s}}\!\!\!\!\!\!{}_{(1)}\!\langle \mathbb{E}[D_r\partial^q_{s_1, \cdots, s_q} u^\sharp | \mathcal{F}_r], dB_r\rangle_{\mathbb{R}^{m}},
\end{equation}
hence
\[
\partial^q_{s_1, \cdots, s_q} u^\sharp = \mathbb{E}(\partial^q_{s_1, \cdots, s_q} u^\sharp | \mathcal{F}_{\bar{s}})
+ \int_{\bar{s}}^T\!\!\!{}_{(1)}\!\langle \mathbb{E}[D_r\partial^q_{s_1, \cdots, s_q} u^\sharp | \mathcal{F}_r], dB_r\rangle_{\mathbb{R}^{m}}.
\]
Lemma 2.4 of Nualart and Pardoux \cite{nualart1988stochastic} shows that the conditional expectation of a $\mathbb{D}^{2,1}$ process is again in $\mathbb{D}^{2,1}$.  Applying their result to $u\in\mathbb{D}^{2,1}$ in our case, we have 
$\mathbb{E}(\partial^q_{s_1, \cdots,s_{j-1}, r,s_{j+1}, \cdots, s_q} u^\sharp |\mathcal{F}_{r})\in\mathbb{D}^{2,1},$ 
and almost surely
\[
D_{s_j}\mathbb{E}(\partial^q_{s_1, \cdots,s_{j-1}, r,s_{j+1}, \cdots, s_q} u^\sharp |\mathcal{F}_{r})
= \mathbb{E}(D_{s_j}\partial^q_{s_1, \cdots,s_{j-1}, r,s_{j+1}, \cdots, s_q} u^\sharp |\mathcal{F}_{r})\mathbf{1}_{(s_j, T]}(r).
\]
Therefore, the process 
\[
\int_0^.D_{s_j} \mathbb{E}(\partial^q_{s_1, \cdots,s_{j-1}, r,s_{j+1}, \cdots, s_q} u^\sharp |\mathcal{F}_{r})dr
\] 
is adapted, hence It\^o-integrable. 
A calculation similar to that in the proof of Lemma \ref{lem:commutation_dq} shows that $T_{q-1}(u)\in \dom(d_{q-1}^{\sharp})$, and for a.e. $s_1, \cdots, s_q\in[0,T]$, 
\begin{eqnarray}
\label{eq:CWS_conditioned_CO_qform1}
& & 
\partial^q_{s_1, \cdots, s_{q}}[d_{q-1}^{\sharp} T_{q-1}(u)]\\
&=& 
\sum_{j=1}^q \int_{\max_{i=1, i\neq j}^{q} s_i}^T \!\!\!{}_{(j+1)}\!\langle D_{s_j} \mathbb{E}(\partial^q_{s_1, \cdots,s_{j-1}, r,s_{j+1}, \cdots, s_q} u^\sharp |\mathcal{F}_{r}),dB_{r}\rangle_{\mathbb{R}^{m}}\notag\\
& & 
+ \sum_{j=1}^q \mathbb{E}(\partial^q_{s_1, \cdots, s_q} u^\sharp | \mathcal{F}_{s_j}) \mathbf{1}_{(\max_{i=1, i\neq j}^{q} s_i, T]}(s_j) \notag\\
&=&
\sum_{j=1}^q \int_{\bar{s}}^T \!\!\!{}_{(j+1)}\!\langle \mathbb{E}(D_{s_j} \partial^q_{s_1, \cdots,s_{j-1}, r,s_{j+1}, \cdots, s_q} u^\sharp|\mathcal{F}_{r}),dB_{r}\rangle_{\mathbb{R}^{m}}  \notag\\
& &
+ \mathbb{E}(\partial^q_{s_1, \cdots, s_q} u^\sharp | \mathcal{F}_{\bar{s}}). \notag
\end{eqnarray} 
Subtracting (\ref{eq:CWS_conditioned_CO_qform1}) from (\ref{eq:conditioned_CO_qform}) and making use of equation (\ref{eq:CWS_dq_C}), we obtain 
\begin{eqnarray*}
& &
\partial^q_{s_1, \cdots, s_q} u^\sharp 
- \partial^q_{s_1, \cdots, s_{q}}[d_{q-1}^{\sharp}T_{q-1}(u)] \\
&=& 
\int_{\bar{s}}^T\!\!\!{}_{(1)}\!\langle \mathbb{E}(D_r\partial^q_{s_1, \cdots, s_q} u^\sharp | \mathcal{F}_r), dB_r\rangle_{\mathbb{R}^{m}} \\
& & 
- \sum_{j=1}^q \int_{\bar{s}}^T\!\!\!{}_{(j+1)}\!\langle  \mathbb{E}(D_{s_j} \partial^q_{s_1, \cdots,s_{j-1}, r,s_{j+1}, \cdots, s_q} u^\sharp|\mathcal{F}_{r}),dB_{r}\rangle_{\mathbb{R}^{m}}\\
&=& 
\int_{\bar{s}}^T\!\!{}_{(1)}\!\langle \mathbb{E}[\partial^{q+1}_{r, s_1, \cdots, s_q} (d_q u)^\sharp   |\mathcal{F}_r], dB_r\rangle_{\mathbb{R}^{m}}, 
\end{eqnarray*}
so (\ref{eq:CO_q}) holds for the special case of $u \in \mathbb{D}^{2,1}\Gamma(\wedge^q H)^*$. 
It is also clear from the above calculation that $d_{q-1}^\sharp T_{q-1}(u)\in \mathbb{D}^{2,k}\Gamma(\wedge^{(q-1)} H)$ if $u\in \mathbb{D}^{2,k}\Gamma(\wedge^q H)^*$; i.e.,  $T_{q-1}(u)\in \mathbb{D}^{2,k+1}\Gamma(\wedge^{(q-1)} H)$ if $u\in \mathbb{D}^{2,k}\Gamma(\wedge^q H)^*$. 
\par
A general $q$-form $u\in \dom(d_q)$ can be approximated by a sequence of 
$u_j\in\mathbb{D}^{2,1}$ such that $u_j\rightarrow u$ and $d_qu_j\rightarrow d_qu$ in $L^2$.  The above computation shows
\begin{eqnarray*}
& &
\partial^q_{s_1, \cdots, s_q}[d_{q-1}^{\sharp}T_{q-1}(u_j)] \\
\!\!&=&\!\! \partial^q_{s_1, \cdots, s_q}u_j^\sharp
 - 
\int_{\bar{s}}^T\!\!\!{}_{(1)}\!\langle \mathbb{E}[\partial^q_{r, s_1, \cdots, s_q} (d_q u_j)^\sharp|\mathcal{F}_r], dB_r\rangle_{\mathbb{R}^{m}}\\
\!\!&\rightarrow &\!\!
\partial^q_{s_1, \cdots, s_q} u^\sharp 
 - 
\int_{\bar{s}}^T\!\!\!{}_{(1)}\!\langle \mathbb{E}[\partial^{q+1}_{r, s_1, \cdots, s_q} (d_q u)^\sharp |\mathcal{F}_r], dB_r\rangle_{\mathbb{R}^{m}}.
\end{eqnarray*}
Since the map $T_{q-1}: L^2\Gamma(\wedge^q H) \rightarrow L^2\Gamma(\wedge^{q-1} H)$ 
is continuous, 
we also have $T_{q-1}(u_j)\rightarrow T_{q-1}(u)$. 
As $d_{q-1}$ is a closed operator, 
so is $d_{q-1}^{\sharp}$. Therefore, $T_{q-1}(u)\in\dom({d_{q-1}^{\sharp}})$ and (\ref{eq:CO_q}) holds for $u\in \dom(d_q)$.
\end{proof}
\begin{proof}[\textbf{Proof of Theorem \ref{th:CO_q*}}]
We first prove the result for $u \in \mathbb{D}^{2,2}\Gamma(\wedge^q H)^*$, and use an approximation argument to extend to a general $u\in \dom(d_{q-1}^*)$. 
\par
From the skew-symmetry of $u\in L^2\Gamma (\wedge^q H)^*$, we observe that $S_{q+1}(u)$ is the full skew-symmetrisation of the $(q+1)$-tensor 
\[
\int_0^.\!\!\cdots\!\!\int_0^.\mathbb{E}[\mathbf{1}_{(\max_{i=1}^{q} r_i, T]}(r)D_r\partial^q_{r_1, \cdots, r_q}u^\sharp|\mathcal{F}_{r}] \,dr dr_1\cdots dr_{q}, 
\]
so indeed $S_{q+1}(u)\in L^2\Gamma (\wedge^{(q+1)} H)$. 
We compute 
\begin{eqnarray}
\label{eq:CO_qform*_1}
& &\partial^q_{s_1, \cdots,s_q}[d_q^{*\sharp}S_{q+1}(u)] \\
&=& \int_0^T\!\!\!{}_{(1)}\!\langle \mathbb{E}[\mathbf{1}_{(\bar{s}, T]}(r)D_r\partial^q_{s_1, \cdots, s_q} u^\sharp |\mathcal{F}_{r}],dB_{r}\rangle_{\mathbb{R}^{m}} \notag\\
& &  
- \sum_{j=1}^q \mathbf{1}_{s_j=\bar{s}}\int_0^{ s_j}\!\!\!\!{}_{(1)}\!\langle 
\mathbb{E}[ \tau_{1,j+1}(D_{s_j}\partial^q_{s_1, \cdots,s_{j-1}, r,s_{j+1}, \cdots, s_q}u^\sharp)|\mathcal{F}_{s_j}],dB_{r}\rangle_{\mathbb{R}^{m}}\notag\\
&=& \int_{\bar{s}}^T\!\!\!{}_{(1)}\!\langle \mathbb{E}(D_r\partial^q_{s_1, \cdots, s_q} u^\sharp |\mathcal{F}_{r}),dB_{r}\rangle_{\mathbb{R}^{m}} \notag\\
& &  
- \sum_{j=1}^q \mathbf{1}_{s_j=\bar{s}}\int_0^{s_j}\!\!\!\!{}_{(j+1)}\!\langle 
\mathbb{E}(D_{s_j}\partial^q_{s_1, \cdots,s_{j-1}, r,s_{j+1}, \cdots, s_q}u^\sharp|\mathcal{F}_{s_j}),dB_{r}\rangle_{\mathbb{R}^{m}}.\notag
\end{eqnarray}
From our assumption $u \in \mathbb{D}^{2,2}\Gamma(\wedge^q H)^*$, we know that $u\in \dom(d_{q-1}^*)$ and $d_{q-1}^* u \in \mathbb{D}^{2,1}\Gamma(\wedge^{q-1} H)^*$.  Making use of skew-symmetry and Lemma \ref{lem:commutation_dq*}, we see
\begin{eqnarray}
\label{eq:CO_qform*_2}
& &  
\sum_{j=1}^q(-1)^{j-1}\mathbf{1}_{s_j=\bar{s}}\mathbb{E}[D_{s_j} \partial^{q-1}_{s_1, \cdots,\hat{s}_{j}, \cdots, s_q}(d_{q-1}^*u)^\sharp |\mathcal{F}_{s_j}]\\
&=&  
\sum_{j=1}^q(-1)^{j-1}\mathbf{1}_{s_j=\bar{s}}\mathbb{E}[D_{s_j} 
\int_0^T\!\!\!\!{}_{(1)}\!\langle \partial^q_{r,s_1, \cdots,\hat{s}_{j}, \cdots,s_q}u^\sharp,  dB_{r}\rangle_{\mathbb{R}^{m}} |\mathcal{F}_{s_j}]\notag\\
&=& 
\sum_{j=1}^q(-1)^{j-1}\mathbf{1}_{s_j=\bar{s}}\mathbb{E}[(-1)^{j-1}D_{s_j} 
\int_0^T\!\!\!\!{}_{(j)}\!\langle \partial^q_{s_1, \cdots,s_{j-1}, r,s_{j+1}, \cdots,s_q}u^\sharp, dB_{r}\rangle_{\mathbb{R}^{m}} |\mathcal{F}_{s_j}]\notag\\
&=& 
\sum_{j=1}^q\mathbf{1}_{s_j=\bar{s}}\mathbb{E}( \partial^q_{s_1, \cdots,s_q}u^\sharp+ 
\int_0^{T}\!\!\!\!{}_{(j+1)}\!\langle D_{s_j} \partial^q_{s_1, \cdots,s_{j-1}, r,s_{j+1}, \cdots, s_q} u^\sharp, dB_r\rangle_{\mathbb{R}^{m}}|\mathcal{F}_{s_j})\notag\\
&=& 
\mathbb{E}(  \partial^q_{s_1, \cdots,s_q}u^\sharp|\mathcal{F}_{\bar{s}})\notag\\
& &+ \sum_{j=1}^q\mathbf{1}_{s_j=\bar{s}}
\int_0^{s_j}\!\!\!\!{}_{(j+1)}\!\langle \mathbb{E}(D_{s_j} \partial^q_{s_1, \cdots,s_{j-1}, r,s_{j+1}, \cdots, s_q}u^\sharp|\mathcal{F}_{s_j}), dB_r\rangle_{\mathbb{R}^{m}}.\notag
\end{eqnarray}
Now summing up (\ref{eq:CO_qform*_1}) and (\ref{eq:CO_qform*_2}), we conclude, using the equality (\ref{eq:conditioned_CO_qform}), that 
\begin{eqnarray*}
& &
\partial^q_{s_1, \cdots,s_q} [d_q^{*\sharp} S_{q+1}(u)] + 
\sum_{j=1}^q(-1)^{j-1}\mathbf{1}_{s_j=\bar{s}}\mathbb{E}[D_{s_j} \partial^{q-1}_{s_1, \cdots,\hat{s}_{j}, \cdots, s_q} (d_{q-1}^*u )^\sharp  |\mathcal{F}_{s_j}]\\
&=&
\int_{\bar{s}}^T\!\!\!{}_{(1)}\!\langle \mathbb{E}(D_{r}\partial^q_{s_1, \cdots, s_q} u^\sharp |\mathcal{F}_{r}),dB_{r}\rangle_{\mathbb{R}^{m}} 
+
\, \mathbb{E}( \partial^q_{s_1, \cdots,s_q}u^\sharp|\mathcal{F}_{\bar{s}})\\
&=&\partial^q_{s_1, \cdots,s_q} u^\sharp. 
\end{eqnarray*}
This proves (\ref{eq:CO_q*}) for the case of $u\in \mathbb{D}^{2,2}\Gamma(\wedge^q H)^*$.
\par
For a general $u\in \dom(d_{q-1}^*)\subset L^2\Gamma(\wedge^q H)^*$, we can approximate by a sequence of cylindrical $u_j\in\mathbb{D}^{2,2}$ such that 
$u_j\rightarrow u$ and $d_{q-1}^* u_j\rightarrow d_{q-1}^* u$ in $L^2$.  Our earlier observation regarding the smoothing property of the projection onto the space of adapated processes
implies that 
\[
\mathbb{E}[D_{s_j} \partial^{q-1}_{s_1, \cdots,\hat{s}_{j}, \cdots, s_q}(d_{q-1}^*u_j )^\sharp |\mathcal{F}_{s_j}]
\rightarrow\mathbb{E}[D_{s_j} \partial^{q-1}_{s_1, \cdots,\hat{s}_{j}, \cdots, s_q} (d_{q-1}^*u )^\sharp  |\mathcal{F}_{s_j}]
\]
in $L^2$. 
The computation above shows that, in $L^2$,
\begin{eqnarray*}
& & \partial^q_{s_1, \cdots,s_q} [d_q^{*\sharp}S_{q+1}(u_j)]\\
&=&\partial^q_{s_1, \cdots, s_q}u_j^\sharp
- \sum_{j=1}^q(-1)^{j-1}\mathbf{1}_{s_j=\bar{s}}\mathbb{E}[D_{s_j} \partial^{q-1}_{s_1, \cdots,\hat{s}_{j}, \cdots, s_q} (d_{q-1}^*u_j)^\sharp|\mathcal{F}_{s_j}]\\
&\rightarrow &
\partial^q_{s_1, \cdots, s_q} u^\sharp 
- \sum_{j=1}^q(-1)^{j-1}\mathbf{1}_{s_j=\bar{s}}\mathbb{E}[D_{s_j} \partial^{q-1}_{s_1, \cdots,\hat{s}_{j}, \cdots, s_q} (d_{q-1}^*u )^\sharp  |\mathcal{F}_{s_j}].
\end{eqnarray*}
Since the map $u\mapsto S_{q+1}(u)$ is continuous in $L^2$, we see $S_{q+1}(u_j)\rightarrow S_{q+1}(u)$.  As $d_q^*$ is a closed operator, so is $d_q^{*\sharp}$. Therefore, indeed $S_{q+1}(u)\in\dom(d_q^{*\sharp})$ and (\ref{eq:CO_q*}) holds.
\end{proof}
\section{Extension}
\label{sec:extension}
Fang and Franchi \cite{fang1997differentiable} proved that the It\^o map from the classical Wiener space to a path group is a 
differentiable 
isomorphism in the sense of Malliavin.  This allows us to transport our generalised Clark-Ocone formulae to the path group, where they take the same forms as (\ref{eq:CO_q_P_V}) and (\ref{eq:CO_q*_P_V}). 
\par
More precisely, given a compact Lie group $G$ with its bi-invariant metric and its identity element $e$, let $\mathfrak{g}=T_eG$ be the Lie algebra, and $L_g$ and $R_g$ the left and right translations, respectively, by any element of $g\in G$.  We write now $C_0=C_0([0,T]; \mathfrak{g})$ with its Cameron-Martin space $H= L^{2,1}_0([0,T]; \mathfrak{g})$, and denote by $C_e=C_e([0,T]; G)$ the group of continuous paths starting from $e$ with values in $G$.  
The Bismut tangent space 
\[
\mathcal{H}_\sigma = \{TR_{\sigma_t}(h_t): h\in H, t\in[0,T]\} 
\]
is defined for a.e. path $\sigma\in C_e$. 
Let $\{g_t\}_{t\in[0,T]}\subset G$ be the solution, starting at $e$, of the following left-invariant Stratonovich stochastic differential equation (SDE) 
\begin{equation}
\label{eq:sde}
dg_t=TL_{g_t} \circ dB_t, 
\end{equation}
where $B$ is the canonical Brownian motion on $\mathfrak{g}$. 
The It\^o map $\mathcal{I}:C_0\rightarrow C_e$ of the SDE (\ref{eq:sde}) is given by \[
\mathcal{I}(w)_t 
= g_t(w),\quad w\in C_0, t\in [0, T]. 
\]
This is a measure-preserving isomorphism between $(C_0, \mathcal{F}, \gamma)$ and $(C_e, \mathcal{F}^{e}, \mu)$, where the Wiener measure $\mu$ on $C_e$ is the law of $\mathcal{I}$, 
and the natural filtration $\{\mathcal{F}_t^e\}_{t\in[0,T]}$ on $C_e$ is generated by the evaluation map.  
Fang and Franchi \cite{fang1997differentiable} showed that the pull-back $\mathcal{I}^*$ in fact supplies a diffeomorphism between the $H$- and $\mathcal{H}$-differentiable structures of the exterior algebras over $C_0$ and  $C_e$: i.e.,
\[
\mathcal{I}^*d_q = d_q\mathcal{I}^*, \quad \mbox{ and }\quad
\mathcal{I}^*d_{q}^* = d_{q}^*\mathcal{I}^*.
\]
Therefore, differential forms on the path group can be pulled back to those on the Wiener space, where we can apply the generalised Clark-Ocone formulae before transferring them back to the path group.  After modifying Notation \ref{no:partial} by setting 
\[
\partial^q_{r_1, \cdots, r_q} = 
[(TR_{g_{r_1}})\frac{\partial}{\partial r_1 }(TR_{g_{r_1}})^{-1}]\otimes\cdots \otimes 
[(TR_{g_{r_q}})\frac{\partial}{\partial r_q}(TR_{g_{r_q}})^{-1}], 
\]
we can state the following  
\begin{theorem}
The formulae (\ref{eq:CO_q_P_V}) and (\ref{eq:CO_q*_P_V}) hold on $C_e$.
\end{theorem}
Elworthy and Li \cite{elworthy2007ito} introduced a `no redundant noise'  class of examples of Riemannian path spaces and extended the results of Fang and Franchi \cite{fang1997differentiable} to this class, at least in the case of $q=1$, for which an analogue of the above theorem can be stated.  For the vanishing of harmonic one-forms, both the path group and the `no redundant noise'  class are covered by the more general result for Riemannian path spaces in \cite{elworthy2011vanishing}. 
\bibliographystyle{amsplain}
\bibliography{../../wt}        
\end{document}